\documentclass{amsart}
\usepackage{latexsym, amsmath, amsfonts, amssymb, geometry, mathrsfs, fancyhdr, tikz, tikz-cd, verbatim}
\usetikzlibrary{matrix,arrows,decorations.pathmorphing,calc}

\def\R{\mathbb{R}}

\def\E{\mathbb{E}}

\def\s{\sigma}

\def\g{\gamma}
\def\t{\tau}

\def\l{\lambda}

\def\S{\mathbb{S}}

\def\a{\alpha}
\def\b{\beta}

\def\la{\langle}
\def\ra{\rangle}
\def\arccosh{\text{arccosh}}
\def\k{\kappa}
\def\H{\mathbb{H}}

\newtheorem{theorem}{Theorem}[section]
\newtheorem{cor}{Corollary}[theorem]

\newtheorem{lemma}{Lemma}[section]

\theoremstyle{definition}

\theoremstyle{remark}
\newtheorem{remark}{Remark}[section]

\numberwithin{equation}{section}

\begin{document}

\title{Synthetic Geometry in Hyperbolic Simplices}

\author{Andrew Clickard}
\address{Department of Mathematical and Digital Sciences, Bloomsburg University, Bloomsburg, Pennsylvania 17815}
\email{ac24869@huskies.bloomu.edu}

\author{Barry Minemyer}
\address{Department of Mathematical and Digital Sciences, Bloomsburg University, Bloomsburg, Pennsylvania 17815}
\email{bminemyer@bloomu.edu}


\date{\today.}



\begin{abstract}
Let $\t$ be an $n$-simplex and let $g$ be a metric on $\t$ with constant curvature $\k$.  The lengths that $g$ assigns to the edges of $\t$, along with the value of $\k$, uniquely determine all of the geometry of $(\t, g)$.  In this paper we focus on hyperbolic simplices ($\k = -1$) and develop geometric formulas which rely only on the edge lengths of $\t$.  Our main results are distance and projection formulas in hyperbolic simplices, as well as a projection formula in Euclidean simplices. We also provide analogous formulas in simplices with arbitrary constant curvature $\kappa$.

\end{abstract}

\maketitle

\section{Introduction}\label{Introduction}

If one fixes the length of each edge of a triangle, or more generally of an $n$-dimensional simplex, then the geometry of the triangle (simplex) is completely determined: no vertex can be moved without changing the length of some adjacent edge.  
Contrast this to a square, for example, where two adjacent vertices can be shifted in the same direction while not changing the lengths of any edges.  
Thus, if $\t$ is an $n$-simplex and $g$ is a metric on $\t$ with constant curvature, then the geometry of $(\t,g)$ is completely determined by just the edge lengths that $g$ assigns to $\t$.  
What we mean in the title by ``synthetic geometry" is formulas/procedures to compute geometric quantities of $(\t,g)$ using only the edge lengths of $\t$, and without needing to isometrically embed $(\t,g)$ into the appropriate model space and use explicit coordinates.   

The study of {\it Euclidean simplices}, or simplices $(\t,g)$ where $g$ has constant curvature zero, has a long history.  
Heron's formula, which computes the area of a triangle using its edge lengths, is an example of synthetic geometry in a Euclidean 2-simplex which dates back thousands of years.  
The study of higher dimensional simplices goes back (at least) to Menger in \cite{Menger} and Cayley in \cite{Cayley}.  
More recently, the second author in \cite{Minemyer3} developed a more efficient technique to encode the geometry of a Euclidean $n$-simplex using the edge lengths of the simplex (this formula is also developed in \cite{Rivin}).  
The techniques and results in \cite{Minemyer3} will be summarized in Section \ref{section:Euclidean}, as they will play a vital role in the work done in this paper.  

Hyperbolic simplices, or metric simplices $(\t,g)$ where $g$ has constant curvature -1, play an important role in the current mathematical zeitgeist.  
For example, distances in hyperbolic triangles considering only the edge lengths of that triangle are used to determine whether a geodesic metric space is CAT(-1) (or an Alexandrov space with curvature bounded below by $-1$).  
Hyperbolic structures are used by Charney and Davis in \cite{CD} for their strict hyperbolization, which are then further used by Ontaneda in \cite{Ontaneda} in his smooth Riemannian hyperbolization.  

The goal of this research is to develop synthetic geometric formulas for simplices with constant curvature. This, together with \cite{Minemyer3}, establishes a foundation on which further research may be more easily performed due to the formulas' relative simplicity in contrast to previous work in the literature.  
We will focus primarily on {\it hyperbolic simplices} due to their mathematical importance, but the last Section will discuss analogous results for {\it spherical simplices} (constant curvature 1), and more generally for simplices of constant curvature $\k$.  

The main results of this paper are as follows:  
\begin{enumerate}
    \item  We develop a simple criterion which determines whether or not a set of positive edge lengths for an $n$-simplex determine a legitimate hyperbolic simplex.  (Section \ref{section:realizability}, Theorem \ref{theorem : hyperbolic-realizability}).  
    
    \item  Given a hyperbolic simplex $(\t,g_{\H})$ and two points $x, y \in \t$, we determine an easy procedure to find $d_{\H}(x,y)$ using only the edge lengths of $\t$ and the barycentric coordinates of $x$ and $y$.  (Section \ref{section:distances}, Theorem \ref{thm:hyp dist}).  
    
    \item  Given a Euclidean $n$-simplex $(\t,g_{\E})$ where $\t = \la v_1, v_2, \hdots, v_n, v_{n+1} \ra$, we develop a formula for $\text{proj}_{\t_{n+1}}(v_{n+1})$, the orthogonal projection of $v_{n+1}$ onto the $(n-1)$-face opposite of it (denoted by $\t_{n+1}$).  (Section \ref{section:Euclidean proj}, Theorem \ref{thm:Euclid Projection}). 
    
    \item  Given a hyperbolic $n$-simplex $(\t,g_{\H})$ where $\t = \la v_1, v_2, \hdots, v_n, v_{n+1} \ra$, we develop a formula for $\text{proj}_{\t_{1}}^{\H}(v_{1})$.  (Section \ref{section:hyperbolic proj}, Theorem \ref{theorem:hyperbolic-projection}). 
\end{enumerate}

The goal of all of this work is to develop formulas which are simple to use and reasonably intuitive.  
Toward this, in Section \ref{section:example} we work out an example with a 3-simplex in which we use all of the formulas mentioned above (and some of the formulas in \cite{Minemyer3}).  
The hope is that this example will be useful for any researchers who wish to use these formulas in the future. 

Lastly, there are more difficult formulas in the literature for some of the quantities listed above.  
In \cite{KSY}, Karli\u{g}a, Savas, and Yakut give a formula for orthogonal projection in hyperbolic space (Theorem 3 in \cite{KSY}).  
As one can see, our formula is considerably simpler, and in any case their formula uses outward normal vectors and is therefore not a truly ``synthetic" formula.  
Also, in \cite{Karliga}, Karli\u{g}a provides necessary and sufficient conditions for when a collection of edge lengths yields a legitimate hyperbolic simplex.  
But again, one can see that our necessary and sufficient condition listed in Section \ref{section:realizability} is more natural, and simpler to use.  
Also, just before submitting this paper to the arXiv, Abrosimov and Vuong posted the article \cite{AV} which gives a geometric version for our Theorem \ref{theorem : hyperbolic-realizability} when $n=3$.

\section{Notation and Formulas in Euclidean Simplices}\label{section:Euclidean}

Let us first establish some notation for the remainder of the paper.  
Let $\t = \la v_1, v_2, \hdots, v_n, v_{n+1} \ra$ be an $n$-dimensional simplex, and let $g$ be a Riemannian metric on $\t$ with constant curvature.  
The notation $g_\E$ means that $g$ has constant curvature $0$, or is {\it Euclidean}; the notation $g_\H$ implies that $g$ has curvature $-1$, or is {\it hyperbolic}; and the notation $g_\mathbb{S}$ means that $g$ has curvature $1$, or is \textit{spherical}.  
Let $e_{ij}$ denote the edge of $\t$ adjacent to the vertices $v_i$ and $v_j$, and let $\g_{ij}$ denote the length of $e_{ij}$ under $g$, also denoted by $g(e_{ij})$.
We make the convention that $\g_{ii} = 0$ for all $i$.
Denote the determinant of a matrix $M$ by $|M|$, and let $M_{ij}$ denote the $ij$-th minor of $M$ (that is, $M_{ij}$ equals the determinant of the matrix obtained by removing the $i^{th}$ row and $j^{th}$ column of $M$).

The purpose of this Section is to quickly summarize and explain the main results from \cite{Minemyer3} to be used in this paper. 
Suppose a Euclidean $n$-simplex $(\t,g_{\E})$ with $\t = \la v_1, \hdots, v_{n+1} \ra$ is linearly isometrically embedded into $\R^m$ ($m \geq n$) endowed with some symmetric bilinear form $\la , \ra$.  
Translate the image of $v_{n+1}$ so that it is mapped to the origin, and by abuse of notation we associate each vertex $v_i$ with its image in $\R^m$.  
Define the vectors $w_i = v_i - v_{n+1}$ for $1 \leq i \leq n$.  
The collection $( w_1, \hdots, w_n )$ forms a basis for the smallest subspace of $\R^m$ containing $\t$, and the form $\la, \ra$ is completely determined on this subspace by the $n \times n$ matrix $Q$ whose $ij^{th}$ entry is defined by
    $$ q_{ij} = \la w_i, w_j \ra. $$
    
Now, notice that
    \begin{equation*}
    \g_{ij}^2 = \la w_i-w_j, w_i-w_j \ra = \g_{i,n+1}^2 + \g_{j,n+1}^2 - 2 \la w_i, w_j \ra 
    \end{equation*}
and so
    \begin{equation}\label{eqn:Q}
    q_{ij} = \la w_i, w_j \ra = \frac{1}{2} \left( \g_{i,n+1}^2 + \g_{j,n+1}^2 - \g_{ij}^2 \right).
    \end{equation}
This allows one to construct the matrix $Q$ using only the edge lengths of $(\t,g_\E)$.  
Then, theoretically, one should be able to calculate any geometric quantity of $\t$ using $Q$ since it completely determines the geometry of $(\t,g_\E)$.  

The main results from \cite{Minemyer3} which follow from the definition of $Q$ above are:
    \begin{enumerate}
    \item {\bf Realizability of $(\t,g_{\E})$:}  $n(n+1)/2$ positive real numbers $\{ \g_{ij} \}$ are the edge lengths of some Euclidean simplex $(\t, g_{\E})$ if and only if the matrix $Q$ is positive definite.  
    
    \item  {\bf Distances in $(\t,g_{\E})$:}  Let $x, y \in \t$ be such that $x$ has barycentric coordinates $(\a_i)_{i=1}^{n+1}$ and $y$ has barycentric coordinates $(\b_i)_{i=1}^{n+1}$, where we have $\sum_{i=1}^{n+1} \a_i = 1 = \sum_{i=1}^{n+1} \b_i$.  
Then the squared Euclidean distance between $x$ and $y$ can be calculated by the formula
\begin{equation}\label{euclid-dist}
    d_\E^2(x,y) = [x-y]^T Q [x-y]
\end{equation} 
where $[x-y]$ is the vector in $\R^m$ whose $i^{th}$ coordinate is $(\a_i - \b_i)$, with $1 \leq i \leq n$.
    
    \item  {\bf Volume of $(\t, g_{\E}):$}  The $n$-dimensional volume of $(\t, g_{\E})$ is given by
        \begin{equation*}
        \text{Vol}(\t) = \frac{\sqrt{\text{det}(Q)}}{n!}.
        \end{equation*}
    \end{enumerate}


\section{Determining the realizability of hyperbolic simplices}\label{section:realizability}

\subsection*{The hyperboloid model for hyperbolic space}  The goal of this research is to develop geometric formulas independent of the embedding of our hyperbolic simplex $(\t, g_\H)$ into hyperbolic space.  
But we will need to isometrically embed our hyperbolic simplex into some model space for hyperbolic space in order to prove that our formulas are correct.  
We will always use the hyperboloid model for hyperbolic space, and so we establish our notation for this now.  

We will use the notation $\R^{n,1}$ to denote the standard Minkowski space with signature $(n,1)$.  
That is, as a vector space, $\R^{n,1}$ is just $\R^{n+1}$ endowed with the symmetric bilinear form
    \begin{equation*}
    \la x,y \ra = x_1y_1 + x_2y_2 + \hdots + x_ny_n - x_{n+1}y_{n+1},  
    \end{equation*}
where $x = (x_1, \hdots, x_{n+1})$ and $y = (y_1, \hdots, y_{n+1})$.  
The solution set to the equation $\la x, x \ra = -1$ forms a two-sheeted hyperboloid, and the ``upper" sheet (the sheet with $x_{n+1} > 0$) is our model for $n$-dimensional hyperbolic space $\H^n$.  
Given two points $x, y \in \H^n$, the hyperbolic distance $d_\H(x,y)$ is given by
    \begin{equation}\label{eqn:hyperbolic distance}
    d_\H(x,y) = \arccosh(-\la x,y \ra) 
    \end{equation}
and is equivalent to the induced path-metric on $\H^n$.  

\subsection*{Induced flat simplices}
Let $(\t, g_\H)$ be an $n$-dimensional hyperbolic simplex, and assume that it is isometrically embedded in $\H^n$.  
Let $\t = \la v_1, \hdots, v_{n+1} \ra$, and by abuse of notation associate $v_i$ with its image in $\H^n$. 

Identifying $\H^n$ with the upper-half plane model described above, we can consider the convex hull of the vertices $v_1, \hdots, v_{n+1}$ in $\R^{n,1}$.  
This yields an $n$-dimensional simplex which we will call $\s$.  
Note that $\s$, endowed with the quadratic form inherited from $\R^{n,1}$, has curvature $0$.  

\begin{remark}\label{rmk:non positive definite}
Note that this form may or may not be positive-definite.  
For an easy example of where $(\t,g_\H)$ is a legitimate hyperbolic simplex but the quadratic form on $\s$ is not positive definite, consider the three points
    \begin{equation*}
    v_1 = (0,0,1) \qquad v_2 = (0,1,\sqrt{2}) \qquad v_3 = (0,2,\sqrt{5}).
    \end{equation*}
These three points lie on a line in $\H^2$, but their convex hull $\s$ is a triangle in $\R^{2,1}$.  
The plane containing $\s$ (the $yz$-plane) clearly has signature $(1,1)$.  
Now, perturb the point $v_3$ on the hyperboloid so that the three points are no longer colinear in $\H^2$.  
For a sufficiently small perturbation, this will provide a legitimate hyperbolic triangle where the quadratic form associated to the convex hull is not positive-definite. 
\end{remark}

In many of the formulas and arguments later in this paper, we will care about the $(n+1)$-dimensional simplex $\Sigma$ defined as follows.  
Assume we have a hyperbolic simplex $(\t, g_\H)$ isometrically embedded in $\H^n$, and define $\s$ as above.
Then $\Sigma := \{ \vec{0} \} \vee \s$.  
That is, $\Sigma = (v_0, v_1, \hdots, v_{n+1})$ where $v_0 = \vec{0}$ and $v_1, \hdots, v_{n+1}$ are the vertices of $\s$.  
So $\Sigma$ is just the $(n+1)$-dimensional simplex in $\R^{n,1}$ obtained by combining the origin with the vertices of $\s$.  

Using similar notation to Section \ref{section:Euclidean}, we can compute a simple formula for the matrix $Q_\Sigma$ associated to $\Sigma$.  
Let $w_i = v_i - v_0 = v_i$.  
The collection $( w_i )$ forms a basis for $\R^{n,1}$.  
With respect to this basis, the $ij^{th}$ entry of $Q_\Sigma$ is given by
    \begin{equation}\label{eqn:Qij}
    q_{ij} = \la w_i, w_j \ra = \la v_i, v_j \ra = -\cosh(\g_{ij})
    \end{equation}
where $\g_{ij} = d_\H(v_i, v_j)$ is the edge length of $\t$ determined by $g_\H$.  
Note that the last equality in equation \eqref{eqn:Qij} is obtained directly from equation \eqref{eqn:hyperbolic distance}, and that the diagonal entries of $Q_\Sigma$ are all $-1$.  
Lastly, observe that the matrix $Q_\Sigma$ can be constructed from $(\t,g_\H)$ without ever needing to isometrically embed $\t$ in $\H$.

\subsection*{Determining the realizability of hyperbolic simplices}
Let $\{ \g_{ij} \}_{i,j=1}^{n+1}$; $\g_{ij} = \g_{ji}$ be a set of positive real numbers. The purpose of this Subsection is to establish whether this set, when defined as the edge lengths of $(\t,g_\H)$, will form a legitimate simplex in hyperbolic space. 

\begin{theorem}\label{theorem : hyperbolic-realizability}
A set of $n(n+1)/2$ positive real numbers $\{\gamma_{ij}\}_{i,j=1, i < j}^{n+1}$ are the edge lengths of a non-degenerate hyperbolic simplex $(\t,g_\H)$ if and only if the matrix $Q_\Sigma$ defined by equation \eqref{eqn:Qij} has signature $(n,1)$.
\end{theorem}
    \begin{proof}
    If $(\t, g_\H)$ is a hyperbolic simplex, then from the discussion above it is clear that $Q_\Sigma$ will have signature $(n,1)$.  
    
    Conversely, assume that $Q_\Sigma$ has signature $(n,1)$.  Let $( \a_i )_{i=1}^{n+1}$ be a basis for $\R^{n+1}$, and define a symmetric bilinear form $\la, \ra$ on $\R^{n+1}$ by 
        \begin{equation*}
        \la \a_i, \a_j \ra = q_{ij}
        \end{equation*}
    for all $i, j$, and where $q_{ij}$ denotes the $ij^{th}$ entry of $Q_{\Sigma}$.  
    The matrix $Q_\Sigma$ is then the Gram matrix for $(\R^{n+1}, \la, \ra )$, and so the form $\la , \ra$ has signature $(n,1)$.  
    Therefore, $\R^{n+1}$ equipped with the form $\la, \ra$ is a model for Minkowski space $\R^{n,1}$.  
    
    The isometric embedding of $(\t, g_\H)$ into $(\R^{n+1}, \la, \ra)$ is just the map that sends $v_i$ to the terminal point of $\a_i$ for each $i$.  
    This map is a linear isometry by construction, and every vertex lies on the two-sheeted hyperboloid defined by the equation $\la x, x \ra = -1$.  
    The only remaining thing that needs to be checked is that each vertex is mapped to the same sheet of this hyperboloid.  
    But suppose the vertices $v_i$ and $v_j$ were mapped to opposite sheets of the hyperboloid.  
    Then the vector $v_i - v_j$ would be in the light cone, and therefore we would have that 
        \begin{equation*}
        \la v_i -v_j, v_i - v_j \ra < 0.
        \end{equation*}  
    Expanding the left-hand side of this inequality gives 
        \begin{equation*}
        \la v_i -v_j, v_i - v_j \ra = \la v_i, v_i \ra + \la v_j, v_j \ra - 2 \la v_i, v_j \ra = -1 -1 - 2q_{ij} = 2\cosh(\g_{ij}) - 2
        \end{equation*}
    which is always greater than or equal to 0.  
    \end{proof}

\section{Barycentric Coordinates in Simplices of Constant Curvature}\label{section:coordinates}
For this Section let $(\t,g)$ be a non-degenerate simplex with some constant curvature $\k$ possibly different from 0 or -1.
Linearly isometrically embed $\t$ into the model space $\R^{n,1}$, and by abuse of notation let $v_i$ denote the image of $v_i$ under this realization.  
As before, let $\s$ be the $n$-dimensional Euclidean simplex determined by the convex hull of $(v_1, \cdots, v_{n+1})$ in $\R^{n,1}$.

Let $p \in \s$ be given by the barycentric coordinates $(\a_1, \cdots, \a_{n+1})$, where $\sum_{i=1}^{n+1} \a_i = 1$.
The purpose of this Section is to define the correspond point $\tilde{p}$ in $\t$.  
The immediate idea is to project $p$ onto $\t$ from the origin.  
This point $\tilde{p}$ will then naturally depend on how $\t$ was embedded in $\R^{n,1}$.  
In what follows we will give a formula for how to compute $\tilde{p}$ using only the barycentric coordinates $(\a_i)_{i=1}^{n+1}$, the edge lengths $(\g_{ij})_{i,j=1}^{n+1}$, and the vectors $(v_i)_{i=1}^{n+1}$, therefore proving that $\tilde{p}$ is well-defined with respect to the location of the vertices $(v_i)_{i=1}^{n+1}$ in the model space.

Define $\tilde{p}$ by 
    \begin{equation}\label{eqn:curved-barycentric}
         \tilde{p} =  \begin{cases} 
      	\frac{p}{\k^2 \cdot \sqrt{\la p, p \ra }}  & \qquad \text{if } \k > 0 \\
      	p & \qquad \text{if } \k = 0  \\
      	\frac{p}{\k^2 \cdot \sqrt{-\la p, p \ra }} & \qquad \text{if } \k < 0
   	\end{cases}
    \end{equation}
Note that $\tilde{p}$ is in fact the projection (from the origin) of $p$ onto the model space with constant curvature $\k$.  
We see from the above definition of $\tilde{p}$ that it depends on $p$ and $\la p, p \ra$.  
Of course, $p$ is completely determined by its barycentric coordinates and the location of the vertices of $\t$.  
To see that the same is true of $\la p, p \ra$, we compute
    \begin{equation}\label{eqn:p-dot-p}
    \la p, p \ra = \left\la \sum_{i=1}^{n+1} \a_i v_i, \sum_{j=1}^{n+1} \a_j v_j \right\ra = \sum_{i, j=1}^{n+1} \a_i \a_j \la v_i, v_j \ra 
    \end{equation}
where
    \[ \la v_i, v_j \ra =  \begin{cases} 
      	\frac{1}{\sqrt{\k}} \cos( \g_{ij} )  & \qquad \text{if } \k > 0 \\
      	\frac{1}{2} \left( \g_{i,n+1}^2 + \g_{j,n+1}^2 - \g_{ij}^2 \right) & \qquad \text{if } \k = 0  \\
      	\frac{-1}{\sqrt{-\k}} \cosh( \g_{ij} ) & \qquad \text{if } \k < 0
   	\end{cases}
	\]
Thus, the value of $\la p, p \ra$ depends only on the the barycentric coordinates of $p$, as well as the edge lengths of $\t$ as defined by $g$.  
Then since $p = \sum_{i=1}^{n+1} \a_i v_i$ depends only on its barycentric coordinates and the location of the vertices $(v_i)_{i=1}^{n+1}$, we see that the definition of $\tilde{p} \in \t$ above is well-defined with respect to the location of the vertices of $\t$ and the metric $g$. Note that when referring to a specific point in $\t$, it is generally simpler to instead refer to the point's corresponding point in $\s$, so that will be the convention used in the coming sections, including the example.

\section{Distances in hyperbolic simplices}\label{section:distances}
Let $(\t, g_\H)$ be a hyperbolic simplex, and let $x, y \in \t$ with barycentric coordinates $x = (x_1, \hdots, x_{n+1})$ and $y = (y_1, \hdots, y_{n+1})$.  
The purpose of this Section is to give a simple algorithm to compute $d_\H(x,y)$, the hyperbolic distance between the points $x$ and $y$, using only the edge lengths $(\g_{ij})$ associated to $g$ and the barycentric coordinates of $x$ and $y$.  

First, define 
    \begin{equation*}
    \tilde{x} = \frac{x}{\sqrt{-\la x, x \ra}} \qquad \tilde{y} = \frac{y}{\sqrt{-\la y, y \ra}}
    \end{equation*}
where $\la x, x \ra$ and $\la y, y \ra$ can be easily calculated using the matrix $Q_{\Sigma}$ as described in equation \eqref{eqn:p-dot-p}.
Note that, if $\t$ were linearly isometrically embedded in the hyperboloid model, then $x$ and $y$ would lie on the convex hull $\s$ while $\tilde{x}$ and $\tilde{y}$ would denote the corresponding projections of $x$ and $y$ onto the hyperboloid (from the origin).  

From equation \eqref{eqn:hyperbolic distance} we know that the hyperbolic distance from $\tilde{x}$ to $\tilde{y}$ is $\arccosh(-\la \tilde{x}, \tilde{y} \ra )$.  
This quantity corresponds to $d_{\H}(x,y)$.  
That is, $d_{\H}(x,y) = \arccosh(-\la \tilde{x}, \tilde{y} \ra )$.  
We formally state this in the following Theorem.

\begin{theorem}\label{thm:hyp dist}
Let $(\t, g_\H)$ be a hyperbolic simplex, and let $x, y \in \t$ with barycentric coordinates $x = (x_1, \hdots, x_{n+1})$ and $y = (y_1, \hdots, y_{n+1})$. 
Then the hyperbolic distance between $x$ and $y$ is given by
    \begin{equation}\label{eqn: hyp dist formula}
    d_\H (x,y) = \arccosh (-\la \tilde{x}, \tilde{y} \ra) = \arccosh \left( \frac{ -\la x,y \ra }{\sqrt{\la x,x \ra \cdot \la y,y \ra }} \right).
    \end{equation}
\end{theorem}

Finally, note that the inner products in equation \eqref{eqn: hyp dist formula} are very easy to calculate using the matrix $Q_\Sigma$.  
If one defines
    \begin{equation*}
    \vec{x} = \begin{pmatrix}
    x_1 \\ x_2 \\ \vdots \\ x_{n+1}
    \end{pmatrix}
    \qquad
    \text{and}
    \qquad
    \vec{y} = \begin{pmatrix}
    y_1 \\ y_2 \\ \vdots \\ y_{n+1}
    \end{pmatrix}
    \end{equation*}
then
    \begin{equation*}
    \la x, y \ra = \vec{x}^T Q_{\Sigma} \vec{y} \qquad \la x,x \ra = \vec{x}^T Q_\Sigma \vec{x} \qquad \la y, y \ra = \vec{y}^T Q_{\Sigma} \vec{y}.
    \end{equation*}

\section{Orthogonal projection in Euclidean Simplices}\label{section:Euclidean proj}
For the remainder of this paper, we denote the sub-simplex generated by removing $v_i$ from $(\t,g)$ by $(\t_i,g)$. Likewise, the sub-simplex generated by removing $v_i$ and $v_j$ is denoted by $\t_{ij}$, etc. The natural question arises as to the barycentric coordinates of the orthogonal projections of some point $p \in \t$ onto one of these sub-simplices in $(\t,g_\E)$ and $(\t,g_\H)$. Define the matrix $Q_{\t_{n+1}}$ as the form from \cite{Minemyer3} for $(\t_{n+1},g_\E)$ (defined by equation \eqref{eqn:Q}). This Section will be focusing on $(\t,g_\E)$, and the following Section will discuss $(\t,g_\H)$.

We first note that in both cases we need only consider the projection of a vertex onto a sub-simplex, as for any other point $x \in \t$, we may subdivide $\t$ in a manner which makes $x$ into a vertex opposite the face we are projecting onto. Similarly, for projection onto a sub-simplex with $n-2$ or fewer vertices, we need only know how to project onto an $(n-1)$-face, as we may define a new simplex by removing the vertices that are not being projected and are not in the sub-simplex. The question, then, becomes one of optimization: we must minimize the distance from our projective vertex to the $(n-1)$-face. For notational purposes, we relabel our vertices so that we are always projecting $v_{n+1}$ onto the face $\t_{n+1}$. 
Let $p = \text{proj}_{\t_{n+1}}(v_{n+1})$.

To prove our formula in Theorem \ref{thm:Euclid Projection} for the barycentric coordinates of $p$, we first need the following Lemma.  
Recall our notation that for a square matrix $Q$, we denote its determinant by $|Q|$.  

\begin{lemma}\label{determinantEquality}
$|Q_\t|=d^2_\E(v_{n+1},p)|Q_{\t_{n+1}}|$.
\end{lemma}
\begin{proof}
A well known result has that $\text{Vol}(\t)=\frac{\text{Vol}(\t_{n+1})d_\E(v_{n+1},p)}{n}$, but by Theorem 4 of \cite{Minemyer3}, \\$\text{Vol}(\t) = \frac{1}{n!}\sqrt{|Q_\t|}$, and $\text{Vol}(\t_{n+1})=\frac{1}{(n-1)!}\sqrt{|Q_{\t_{n+1}}|}$. So \[\text{Vol}(\t) = \frac{\text{Vol}(\t_{n+1})d_\E(v_{n+1},p)}{n} \iff \frac{1}{n!}\sqrt{|Q_\t|}=\frac{\sqrt{|Q_{\t_{n+1}}|}d_\E(v_{n+1},p)}{n!}\]
Solving for $|Q_\t|$ then, we obtain
\[
|Q_\t|=|Q_{\t_{n+1}}|d^2_\E(v_{n+1},p),
\]thus completing the proof.
\end{proof}

\begin{theorem}\label{thm:Euclid Projection}
Let $(\t,g_\E)$ be a Euclidean $n$-simplex with $\t = ( v_1,v_2,\dots,v_{n+1} )$. Let $\t_{n+1}$ be the $(n-1)$-face of $\t$ with vertices $( v_1,\dots,v_n )$. Then the barycentric coordinates of the orthogonal projection $p$ of $v_{n+1}$ onto $\t_{n+1}$ are given by:
\[
\a_i = \frac{\sum_{j=1}^n (-1)^{i+j}Q_{ij}} {|Q_{\t_1}|};\; 1\leq i \leq n.
\]
where $p = (\a_1, \dots, \a_{n},0)$, $Q = Q_\t$, and $Q_{ij}$ denotes the $ij^{th}$ minor of $Q$.
\end{theorem}

    \begin{proof}
    Linearly isometrically embed the Euclidean simplex $(\t, g_\E)$ into $\R^n$ in some way, and by abuse of notation identify each vertex $v_i$ with its image under this isometry.  
    Let $w_i = v_i - v_{n+1}$.  
    Then the collection $(w_1, \dots, w_n)$ forms a basis for $\R^n$.  
    
    We proceed via the method of Lagrange multipliers, and we seek to minimize $d_\E(v_{n+1},p)$ subject to the constraint $\a_1 + \a_2 + \a_3 + \cdots + \a_n=1$. Let $\Vec{\a} = \begin{bmatrix} \a_1 & \a_2 & \cdots & \a_n  \end{bmatrix}^T$. Then $d_\E^2(v_{n+1},p) = \Vec{\a}^TQ\Vec{\a}$ by Equation \eqref{euclid-dist}.  Calculating the distance function, then, we obtain that \[d_\E^2(v_{n+1},p)=\sum\limits_{i,j=1}^n q_{ij}\a_i\a_j,\] (where $q_{ij}$ denotes the $ij^{th}$ entry of $Q$) noting here that due to symmetry there are exactly two copies of each term where $i\neq j$. We now define our Lagrangian function: \[\mathscr{L}(\a_1,\a_2,\cdots,\a_n,\l) = \sum\limits_{i,j=1}^n q_{ij}\a_i\a_j - \l\left(\left(\sum\limits_{i=1}^n\a_i\right) -1\right)\]
    We now seek to optimize $\mathscr{L}$. Consider just one $\frac{\partial}{\partial\a_i}\mathscr{L}$. The nonconstant components of $\mathscr{L}$ with respect to $\a_i$ are \[2q_{i1}\a_1\a_i + 2q_{i2}\a_2\a_i + \cdots + q_{ii}\a_i^2 + \cdots + 2q_{in}\a_n\a_i - \l\a_i,\] so
    \[\frac{\partial}{\partial\a_i}\mathscr{L} = \left(\sum\limits_{j=1}^n  2q_{ij}\a_j\right) - \l \text{, and thus, the gradient of $\mathscr{L}$ is }
    \nabla\mathscr{L}=\begin{pmatrix}
    \left(\sum\limits_{j=1}^n  2q_{1j}\a_j\right) - \l\\
    \left(\sum\limits_{j=1}^n  2q_{2j}\a_j\right) - \l\\
    \vdots\\
    \left(\sum\limits_{j=1}^n  2q_{nj}\a_j\right) - \l
    \end{pmatrix}:=0.
    \]Adding across by the $n\times 1$ column vector with $\l$ as entries, we obtain
    \begin{equation}\label{lambda_equation}
    \begin{pmatrix}
    \sum\limits_{j=1}^n  2q_{1j}\a_j\\
    \sum\limits_{j=1}^n  2q_{2j}\a_j\\
    \vdots\\
    \sum\limits_{j=1}^n  2q_{nj}\a_j
    \end{pmatrix}=\begin{pmatrix}\l \\ \l \\ \vdots \\ \l\end{pmatrix}=\vec{\l}.\end{equation}
    From which we may factor out $\Vec{\a}$ and obtain
    \[\begin{pmatrix}
    2q_{11} & 2q_{12} & \cdots & 2q_{1n} \\
    2q_{12} & 2q_{22} & \cdots & 2q_{2n} \\
    \vdots & \vdots & \ddots & \vdots \\
    2q_{1n} & \cdots & \cdots & 2q_{nn} \\
    \end{pmatrix}\vec{\a}=\begin{pmatrix}\l \\ \l \\ \vdots \\ \l\end{pmatrix}.\] But clearly, this $n\times n$ matrix is $2Q$. Since $\t$ is a nondegenerate simplex, $Q$ is positive definite, and is thus invertible. Left multiplying by $Q^{-1}$, we obtain: $2\vec{\a}=Q^{-1}\vec{\l}.$ But $Q^{-1}=\frac{1}{|Q|}C$, where $C$ is the cofactor matrix of $Q$. Thus, 
    \[
    \vec{\a}=\frac{1}{2|Q|}C\vec{\l} = \frac{\l}{2|Q|}\begin{pmatrix}
    \sum_{i=1}^n (-1)^{i+1}Q_{1i}\\
    \sum_{i=1}^n (-1)^{i+2}Q_{2i}\\
    \vdots\\
    \sum_{i=1}^n (-1)^{i+n}Q_{ni}
    \end{pmatrix}
    \]Then by Lemma \ref{determinantEquality}, we obtain
    \[\vec{\a}=\frac{\l}{2|Q_{\t_{n+1}}|d^{2}_\E(v_{n+1},p)}\begin{pmatrix}
    \sum_{i=1}^n (-1)^{i+1}Q_{1i}\\
    \sum_{i=1}^n (-1)^{i+2}Q_{2i}\\
    \vdots\\
    \sum_{i=1}^n (-1)^{i+n}Q_{ni}
    \end{pmatrix}.\]
    Now let us inspect Equation \eqref{lambda_equation}.
    Left multiplying both sides by $\vec{\a}^T$, we obtain $2d^{2}(v_{n+1},p) = \l\left(\sum_{i=1}^n \a_i\right)$.
    But $\sum_{i=1}^n\a_i = 1$ by the constraint, so $\l = 2d_\E^2(v_{n+1},p)$. Thus, we finally have
    \[
    \vec{\a}=\begin{pmatrix}
    \frac{\sum_{i=1}^n (-1)^{i+1}Q_{1i}}{|Q_{\t_1}|}\\
    \frac{\sum_{i=1}^n (-1)^{i+2}Q_{2i}}{|Q_{\t_1}|}\\
    \vdots\\
    \frac{\sum_{i=1}^n (-1)^{i+n}Q_{ni}}{|Q_{\t_1}|}
    \end{pmatrix},\]as desired.
    
    \end{proof}
    
\begin{cor}\label{sum-determinant equality}
    $\sum\limits_{i,j=1}^n (-1)^{i+j}Q_{ij} = |Q_{\t_{n+1}}|$.
\end{cor}
    \begin{proof}
    By Theorem \ref{thm:Euclid Projection} and the definition of barycentric coordinates in Euclidean simplices, we have:
    \[ 1 = \sum\limits_{i=1}^n \a_i =\sum\limits_{i=1}^n \frac{\sum_{j=1}^n (-1)^{i+j}Q_{ij}} {\text{det}(Q_{\t_{n+1}})}= \frac{1}{\text{det}(Q_{\t_{n+1}})}\sum\limits_{i,j=1}^n (-1)^{i+j}Q_{ij}.\] Thus, we have \begin{equation}\label{minor-det eqn}
    \sum\limits_{i,j=1}^n (-1)^{i+j}Q_{ij} = \text{det}(Q_{\t_{n+1}}).
    \end{equation}
    \end{proof}
    
\begin{cor}
    The barycentric coordinates of the orthogonal projection $p=(0,\a_1,\cdots,\a_n)$ of $v_{n+1}$ onto $\t_{n+1}$ are given by\[\a_i = \frac{\sum_{j=1}^{n}(-1)^{i+j}Q_{ij}}{\sum_{j,k=1}^{n}(-1)^{j+k}Q_{jk}}.\]
\end{cor}
    \begin{proof}
    Immediate from Theorem \ref{thm:Euclid Projection} and Corollary \ref{sum-determinant equality}.
    \end{proof}
    
\begin{cor}
    $\text{Vol}(\t_{n+1}) = \frac{1}{(n-1)!}\sqrt{\sum\limits_{i,j=1}^n (-1)^{i+j}Q_{ij}}$.
\end{cor}
    \begin{proof}
    Theorem 4 of \cite{Minemyer3} shows that $\text{Vol}(\t_{n+1})=\frac{1}{(n-1)!}\sqrt{|Q_{\t_{n+1}}|}$. Substituting Equation \eqref{minor-det eqn} for the radical, we obtain
    \[ \text{Vol}(\t_{n+1}) = \frac{1}{(n-1)!}\sqrt{\sum\limits_{i,j=1}^n (-1)^{i+j}Q_{ij}}, \] thus completing the proof.
    \end{proof}

\section{Orthogonal Projection in Hyperbolic Simplices}\label{section:hyperbolic proj}

 We now turn our attention to orthogonal projection in hyperbolic simplices.
 In this Section we are going to project the vertex $v_1$ onto the face $\t_1$, as opposed to last Section where we projected $v_{n+1}$ onto $\t_{n+1}$.  
 This reason for this change in notation is purely to make labeling subscripts easier.
 
 A first thought may be to just project within the convex hull $\s$ of the vertices of $\t$, and then poject this point onto the hyperboloid.  But, in general, this does not work.  The reason for this can be found in Remark \ref{rmk:non positive definite}.  The quadratic form in $\R^{n,1}$ when restricted to the hyperplane containing the vertices of $\t$ may not be positive definite.  In that case, the procedure in Theorem \ref{thm:Euclid Projection} may not work.  

\begin{theorem}\label{theorem:hyperbolic-projection}
Let $(\t,g_\H)$ be a hyperbolic $n$-simplex with vertices $( v_1,v_2,\cdots,v_{n+1} )$. Then the barycentric coordinates of the orthogonal projection $p \in \t$ of $v_1$ onto $\t_1$ is given by  
    \begin{equation}\label{eqn:hyperbolic proj coords}
    \a_i = \frac{(-1)^{i+1}Q^\Sigma_{1i}}{\sum_{j=2}^{n+1}(-1)^{1+j}Q_{1j}^\Sigma}.
    \end{equation}
where $p = (0,\a_2,\a_3,\cdots,\a_{n+1})$, $Q_\Sigma$ is defined as in equation \eqref{eqn:Qij}, and $Q_{ij}^\Sigma$ denotes the $ij^{th}$ minor of $Q_{\Sigma}$.  
\end{theorem}

\begin{remark}
    Suppose the hyperbolic simplex $(\t, g_\H)$ is linearly isometrically embedded in the hyperboloid model for $\H^n$.  The point $p$ described in Theorem \ref{theorem:hyperbolic-projection} would lie in the convex hull $\s$ of $\t$ (it would actually lie on the covex hull of the points $(v_2, \dots, v_{n+1})$).  
    To find the actual point $\tilde{p}$ that is the orthogonal projection of $v_1$ onto $\t_1$, you would need to project $p$ onto the hyperboloid from the origin.  To do this, recall that $\tilde{p} = p/\sqrt{- \la p, p \ra }$.  
    A formula for the components of $\tilde{p}$ is given by
    \[\tilde{p}=(0,\tilde{\a}_2,\tilde{\a}_3,\cdots,\tilde{\a}_{n+1}),\text{ where  }\tilde{\a}_i = \frac{(-1)^{i+1}Q^\Sigma_{1i}}{\sqrt{\sum_{i,j=1}^{n+1}(-1)^{i+j}Q^\Sigma_{1i}Q^\Sigma_{1j}(\cosh(\g_{ij}))}},\]
    In practice though, it is easier to calculate $\a_i$ using equation \eqref{eqn:hyperbolic proj coords}, calculating $\sqrt{- \la p , p \ra}$ using $Q_\Sigma$, and then dividing.  
\end{remark}

\begin{remark}
    Note that the formula for calculating the orthogonal projection in Theorem \ref{thm:Euclid Projection} uses the $n \times n$ matrix $Q_\t$ defined in equation \eqref{eqn:Q}, whereas the formula in Theorem \ref{theorem:hyperbolic-projection} uses the $(n+1) \times (n+1)$ matrix $Q_\Sigma$ defined in \eqref{eqn:Qij}.  
    This difference is why it is notationally easier to project $v_{n+1}$ onto $\t_{n+1}$ in Theorem \ref{thm:Euclid Projection} and $v_1$ onto $\t_1$ in Theorem \ref{theorem:hyperbolic-projection}.
\end{remark}

    \begin{proof}[Proof of Theorem \ref{theorem:hyperbolic-projection}]
    Linearly isometrically embed the hyperbolic simplex $(\t, g_\H)$ into the hyperboloid model for $\H^n$.  
    By abuse of notation, identify each vertex $v_i$ with its image in the hyperboloid.  
    Let $\s$ denote the convex hull of the vertices $(v_1, \dots, v_{n+1})$, and let $\Sigma$ be the $(n+1)$-simplex obtained as the convex hull of $(v_0, v_1, \dots, v_{n+1})$ where $v_0$ is the origin in $\R^{n,1}$.  
    
    Just as in the proof of Theorem \ref{thm:Euclid Projection} we proceed via the method of Lagrange Multipliers. Let $p = \a_2v_2 + \a_3v_3 + \cdots + \a_{n+1}v_{n+1}$ be the point on $\s$ corresponding to the projection $\tilde{p}$ of $v_1$ onto $\t_1$. Then $\tilde{p} = \frac{p}{\sqrt{-\la p,p \ra}}$ by Equation \eqref{eqn:curved-barycentric}. 
    The hyperbolic distance between $v_1$ and $\tilde{p}$ is given by
        \begin{equation*}
        d_\H(v_1, \tilde{p}) = \arccosh(- \la v_1, \tilde{p} \ra).
        \end{equation*}
    Since $\arccosh()$ is increasing for arguments greater than $1$, our goal is to maximize $\la v_1,\tilde{p}\ra$ subject to the constraint $\a_2+\a_3+\cdots + \a_{n+1}=1$.
    
    Define our Lagrangian function $\mathscr{L}(\a_2,\cdots,\a_{n+1},\l):=\la v_1,\tilde{p} \ra -\l(\a_2+\dots+\a_{n+1}-1)$. Consider just one $\frac{\partial}{\partial \a_i}\mathscr{L}:$
    
    \[
    \frac{\partial}{\partial \a_i}\mathscr{L} = \la v_1,\frac{\partial}{\partial \a_i}\tilde{p}\ra - \l = \left\la v_1, \frac{-1}{2}(-\la p,p\ra)^{-3/2}(-2\la p,v_i\ra)p + (-\la p,p \ra)^{-1/2}v_i \right\ra
    -\l\]
    \[
    = \frac{\la p,v_i\ra}{(-\la p,p \ra)^{3/2}}\la v_1,p \ra + \frac{\la v_1,v_i \ra}{(-\la p,p\ra)^{1/2}} - \l := 0
    \] 
    By adding over the $\l$ and clearing denominators, we obtain
    \[
    \frac{\la p,v_i\ra}{-\la p,p \ra}\la v_1,p \ra + \la v_1,v_i \ra = \l(-\la p,p \ra)^{1/2}
    \]
    And thus our system of equations can be written as
    \begin{align}
    \frac{\la p,v_2\ra}{-\la p,p \ra}\la v_1,p \ra + \la v_1,v_2 \ra &=  \l(-\la p,p \ra)^{1/2}  \label{eqn:2} \\
    \frac{\la p,v_3\ra}{-\la p,p \ra}\la v_1,p \ra + \la v_1,v_3 \ra &=  \l(-\la p,p \ra)^{1/2} \label{eqn:3} \\
    \vdots \notag \\
    \frac{\la p,v_{n+1}\ra}{-\la p,p \ra}\la v_1,p \ra + \la v_1,v_{n+1} \ra &=  \l(-\la p,p \ra)^{1/2}  \label{eqn:n+1}
    \end{align}
    The first step to solving this system of equations is to show that $\l = 0$.  
    To do this, we take $\a_2 \eqref{eqn:2} + \a_3 \eqref{eqn:3} + \dots + \a_{n+1} \eqref{eqn:n+1}$.
    Recalling the constraint $\sum_{i=2}^{n+1}\a_i=1$, we obtain:
    
    \[
    \text{Right-Hand Side:} \quad \sum_2^{n+1}\a_i\l(-\la p,p\ra)^{1/2}=\l(-\la p,p\ra)^{1/2}
    \]\[
    \text{Left-Hand Side:} \quad \sum_{2}^{n+1}\a_i\left(\frac{\la p,v_i\ra}{-\la p,p \ra}\la v_1,p \ra + \la v_1,v_i \ra\right)\]\[ = \sum_{i=2}^{n+1}\frac{\la p,\a_iv_i \ra}{-\la p,p \ra}\la v_1,p\ra + \la v_1, \a_iv_i\ra = \frac{\la p,p\ra}{-\la p,p\ra}\la v_1,p\ra+\la v_1,p\ra = 0
    \] And thus since $(-\la p,p \ra)^{1/2}\neq 0$, we must have that $\l = 0$. Our system of equations becomes
    \[
    \begin{pmatrix}
    \frac{\la p,v_2\ra}{-\la p,p \ra}\la v_1,p \ra + \la v_1,v_2 \ra\\
    \frac{\la p,v_3\ra}{-\la p,p \ra}\la v_1,p \ra + \la v_1,v_3 \ra\\
    \vdots\\
    \frac{\la p,v_{n+1}\ra}{-\la p,p \ra}\la v_1,p \ra + \la v_1,v_{n+1} \ra
    \end{pmatrix} = \vec{0}
    \] But $-\la p,p \ra = \left(\sqrt{-\la p,p \ra}\right)^2$, and so we may distribute up into our inner products to obtain
    
    \[\begin{pmatrix}
    \la \tilde{p},v_2 \ra \la v_1,\tilde{p} \ra + \la v_1,v_2 \ra\\
    \la \tilde{p},v_3 \ra \la v_1,\tilde{p} \ra + \la v_1,v_3 \ra\\
    \vdots\\
    \la \tilde{p},v_{n+1} \ra \la v_1,\tilde{p} \ra + \la v_1,v_{n+1} \ra
    \end{pmatrix} = \vec{0}
    \]
    Now we may rearrange using the linearity of the bilinear form to achieve
    \[\begin{pmatrix}
    \la \la v_1,\tilde{p}\ra \tilde{p} + v_1, v_2 \ra\\
    \la \la v_1,\tilde{p}\ra \tilde{p} + v_1, v_3 \ra\\
    \vdots\\
    \la \la v_1,\tilde{p}\ra \tilde{p} + v_1, v_{n+1} \ra
    \end{pmatrix}=\vec{0}
    \]
    Therefore we must have that $\la v_1,\tilde{p}\ra \tilde{p} + v_1 \in \bigcap_{i=2}^{n+1}v_i^\perp$. 
    This intersection is one-dimensional, and is spanned by the first column of $Q_{\Sigma}^{-1}$ (for a proof, see Lemma \ref{lemma:orth-complement} below).  
    Since $\tilde{p}$ has a $v_1$ component of 0, we have

    \[
    \la v_1,\tilde{p}\ra \tilde{p} + v_1 = \begin{pmatrix}
    1\\
    \frac{-Q^\Sigma_{12}}{Q^\Sigma_{11}}\\
    \vdots\\
    \frac{(-1)^{(n+1)+1}Q^\Sigma_{1(n+1)}}{Q^\Sigma_{11}}
    \end{pmatrix},
    \text{ which then implies that } 
    \la v_1,\tilde{p}\ra \tilde{p}=\begin{pmatrix}
    0\\
    \frac{-Q^\Sigma_{12}}{Q^\Sigma_{11}}\\
    \vdots\\
    \frac{(-1)^{(n+1)+1}Q^\Sigma_{1(n+1)}}{Q^\Sigma_{11}}\end{pmatrix}.
    \] Factoring $(-\la p,p \ra)^{1/2}$ from each of the $\tilde{p}$'s and dividing by $\frac{\la v_1,p\ra}{-\la p,p \ra}$, we have
    
    \begin{equation}\label{eqn:p-vector}
    p = 
    \begin{pmatrix}
    0\\
    \frac{Q^\Sigma_{12}\la p,p \ra}{Q^\Sigma_{11}\la v_1,p\ra}\\
    \vdots\\
    \frac{(-1)^{n+3}Q^\Sigma_{1(n+1)}\la p,p \ra}{Q^\Sigma_{11}\la v_1,p\ra}
    \end{pmatrix}.
    \end{equation}
    Recall that $p = (0, \a_2, \dots, \a_{n+1})$.  Thus we obtain a preliminary solution for each $\a_i$ by aligning the components of equation \eqref{eqn:p-vector}.  But this is not a sufficient solution since $p$ is needed to compute each $\a_i$.  But consider the sum of these equations:
    
    \[
    1 =\sum_{i=2}^{n+1}\a_i = \frac{-\la p,p \ra}{\la v_1,p \ra}\frac{\sum_{i=2}^{n+1}(-1)^{i+1}Q^\Sigma_{1i}}{Q^\Sigma_{11}},
    \] and thus
    \begin{equation}\label{eqn:p-dot-p-constant}
    \frac{-\la p,p \ra}{\la v_1,p \ra}= \frac{Q^\Sigma_{11}}{\sum_{i=2}^{n+1}(-1)^{i+1}Q^\Sigma_{1i}}.
    \end{equation}
    Substituting into Equation \ref{eqn:p-vector}, we therefore have
    
    \[p = 
    \begin{pmatrix}
    0\\
    \frac{-Q^\Sigma_{12}}{\sum_{i=2}^{n+1}(-1)^{i+1}Q^\Sigma_{1i}}\\\
    \\
    \frac{Q^\Sigma_{13}}{\sum_{i=2}^{n+1}(-1)^{i+1}Q^\Sigma_{1i}}\\
    \vdots\\
    \frac{(-1)^{(n+1)+1}Q^\Sigma_{1(n+1)}}{\sum_{i=2}^{n+1}(-1)^{i+1}Q^\Sigma_{1i}}\\
    \end{pmatrix},
    \] and thus $\a_i = \frac{(-1)^{i+1}Q^\Sigma_{1i}}{\sum_{i=2}^{n+1}(-1)^{i+1}Q^\Sigma_{1i}}$, proving the Theorem. To calculate a formula for $\tilde{p} = \frac{p}{\sqrt{-\la p,p \ra}}$, we must calculate $\sqrt{-\la p,p \ra}$ with respect to $Q_\Sigma$:
    
    \[ \sqrt{-\la p,p \ra} = \sqrt{-\sum_{i,j=2}^{n+1} \a_i\a_j\la v_i,v_j \ra} =\sqrt{\frac{\sum_{i,j=2}^{n+1}(-1)^{i+j+4}Q^\Sigma_{1i}Q^\Sigma_{1j}(\cosh(\g_{ij}))}{\left(\sum_{i=2}^{n+1}(-1)^{i+1}Q^\Sigma_{1i}\right)^2}}\]
    \[ =\frac{\sqrt{\sum_{i,j=2}^{n+1}(-1)^{i+j}Q^\Sigma_{1i}Q^\Sigma_{1j}(\cosh(\g_{ij}))}}{\sum_{i=2}^{n+1}(-1)^{i+1}Q^\Sigma_{1i}}. \]
    Thus, we finally have that
    \[\tilde{p} = \frac{(\sum_{i=2}^{n+1}\a_iv_i)(\sum_{i=2}^{n+1}(-1)^{i+1}Q^\Sigma_{1i})}{\sqrt{\sum_{i,j=2}^{n+1}(-1)^{i+j}Q^\Sigma_{1i}Q^\Sigma_{1j}(\cosh(\g_{ij}))}} = \frac{\sum_{i=2}^{n+1}(-1)^{i+1}Q^\Sigma_{1i}v_i}{\sqrt{\sum_{i,j=2}^{n+1}(-1)^{i+j}Q^\Sigma_{1i}Q^\Sigma_{1j}(\cosh(\g_{ij}))}},\]
    and so $\tilde{\a}_i = \frac{(-1)^{i+1}Q^\Sigma_{1i}}{\sqrt{\sum_{i,j=2}^{n+1}(-1)^{i+j}Q^\Sigma_{1i}Q^\Sigma_{1j}(\cosh(\g_{ij}))}}$, as desired.
    \end{proof}

\begin{lemma}\label{lemma:orth-complement}
Using the notation in Theorem \ref{theorem:hyperbolic-projection}, the intersection of the orthogonal complements of the vertex vectors $v_i$ for $2\leq i \leq n+1$ is 
\[\bigcap_{i=2}^{n+1}v_i^\perp = \left\la\begin{pmatrix}
Q_{11}^\Sigma \\ -Q_{12}^\Sigma \\ \vdots \\ (-1)^{(n+1)+1}Q_{1(n+1)}^\Sigma\\
\end{pmatrix}\right\ra,\] where $\la \vec{v} \ra$ denotes the span of $\vec{v}$.
\end{lemma}
    \begin{proof}
    We first note that $\bigcap_{i=2}^{n+1}v_i^\perp$ is a one dimensional vector space. Let $\vec{x}$ be an $(n+1)\times 1$ column vector in $\bigcap_{i=2}^{n+1}v_i^\perp$. We also note that since we are working with respect to the vertex vectors, \[\la v_i, \vec{x} \ra = \sum_{j=1}^{n+1} \la v_i, x_jv_j\ra = \sum_{j=1}^{n+1} x_j\la v_i,v_j\ra = \sum_{j=1}^{n+1} x_jq_{ij},\] where $x_j$ is the $j$-th entry of $\vec{x}$. Because $\bigcap_{i=2}^{n+1}v_i^\perp$ is one dimensional and does not lie on the hyperplane $x_1 = 0$, we may let $x_1=1$ without loss of generality (and this is the form of $x$ that was needed in Theorem \ref{theorem:hyperbolic-projection}). Then we must solve for $n$ unknowns in $n$ equations:
    \[
    \begin{pmatrix}
    q_{12} + q_{22}x_2 + q_{23}x_3 + \cdots + q_{2(n+1)}x_{n+1}\\
    q_{13} + q_{23}x_2 + q_{33}x_3 + \cdots + q_{3(n+1)}x_{n+1}\\
    \vdots\\
    q_{1(n+1)} + q_{2(n+1)}x_2 + q_{23}x_3 + \cdots + q_{(n+1)(n+1)}x_{n+1}\\
    \end{pmatrix}=\vec{0},
    \] which can be rearranged to:
    \[
    \begin{pmatrix}
    q_{22} & q_{23} & \cdots & q_{2(n+1)} \\
    q_{23} & q_{33} & \cdots & q_{3(n+1)} \\
    \vdots &\vdots& \ddots & \vdots \\
    q_{2(n+1)} & q_{3(n+1)} & \cdots & q_{(n+1)(n+1)} \\
    \end{pmatrix}
    \begin{pmatrix}
    x_2\\x_3\\\vdots\\x_{n+1}
    \end{pmatrix}=
    \begin{pmatrix}
    -q_{12}\\-q_{13}\\\vdots\\-q_{1(n+1)}
    \end{pmatrix}.
    \]
    First note that this matrix is the submatrix $Q_\Sigma(1,1)$ of $Q_\Sigma$ formed by removing the first row and first column. By Cramer's rule, $x_i = \frac{|A_{i-1}|}{Q^\Sigma_{11}}$, and $A_{i-1}$ is the matrix formed by replacing the $(i-1)$th column of $Q(1,1)$ with $\begin{bmatrix}
    -q_{12} & -q_{13} & \cdots & -q_{1(n+1)}
    \end{bmatrix}^T$. Thus, we have that
    \[\vec{x}=\begin{pmatrix}
    1\\
    \frac{-Q_{12}}{Q_{11}}\\
    \vdots\\
    \frac{(-1)^{n+1}Q_{1(n+1)}}{Q_{11}}
    \end{pmatrix}.\]
    And thus, since $\vec{x} \in \bigcap_{i=2}^{n+1}v_i^\perp$ and $\bigcap_{i=2}^{n+1}v_i^\perp$ is one dimensional, $\bigcap_{i=2}^{n+1}v_i^\perp = \la \vec{x} \ra$. Scaling $\vec{x}$ by $Q_{11}$, we have
    \[\bigcap_{i=2}^{n+1}v_i^\perp = \left\la\begin{pmatrix}
Q_{11} \\ -Q_{12} \\ \vdots \\ (-1)^{n+1}Q_{1(n+1)}\\
\end{pmatrix}\right\ra,\] as desired.
    \end{proof}

Symmetrically to the Euclidean Projection, in order to project onto an $n-2$ or smaller sub-simplex, we redefine a simplex that removes the irrelevant vertices and perform this calculation. As one can see, this calculation proves to be far simpler than other methods, and it only relies on the edge lengths of the simplex.

\section{An example}\label{section:example}

The formulas in our Theorems, specifically Theorems \ref{thm:Euclid Projection} and \ref{theorem:hyperbolic-projection}, look a lot more complicated to use than they are in practice.  
The purpose of this Section is to work out an example to demonstrate the efficiency of these formulas.  

Let $\t = \la v_1,v_2,v_3,v_4\ra$ be a 3-simplex with edge lengths given by:
\begin{tabular}{c|c|c|c|c}
    $\g_{ij}$ & 1 & 2 & 3 & 4 \\
    \hline
    1       & 0 & 2 & 3 & 4 \\
    \hline
    2       & 2 & 0 & 4 & 5 \\
    \hline
    3       & 3 & 4 & 0 & 3 \\
    \hline
    4       & 4 & 5 & 3 & 0 \\
\end{tabular}

In this section, we use the methods developed in this paper to compute various quanities in $\t$. We first verify that $\t$ is a non-degenerate simplex when given both the Euclidean metric $g_\E$ and the hyperbolic metric $g_\H$. We then follow this by finding the distances between $p = (\frac{1}{4},\frac{1}{4},\frac{1}{4},\frac{1}{4})$ and $q = (\frac{1}{3},\frac{1}{3},\frac{1}{3},0)$ in the Euclidean and Hyperbolic metric, and finally to find the orthogonal projection of $v_1$ onto $\t_1$ with $\t$ viewed as both Euclidean and Hyperbolic simplex.

\subsection{Calculations in $(\t,g_\E)$}

\subsection*{Verifying that $(\t, g_\E)$ is a legitimate Euclidean simplex}.  We construct the matrix $Q$ from \eqref{eqn:Q}.
\[ 
Q = 
\begin{pmatrix}
1/2(\g_{12}^2 +\g_{12}^2-\g_{22}^2) & 1/2(\g_{12}^2 +\g_{13}^2-\g_{23}^2) & 1/2(\g_{12}^2 +\g_{14}^2-\g_{24}^2) \\ \\
1/2(\g_{12}^2 +\g_{13}^2-\g_{23}^2) & 1/2(\g_{13}^2 +\g_{13}^2-\g_{33}^2) & 1/2(\g_{13}^2 +\g_{14}^2-\g_{34}^2) \\ \\
1/2(\g_{12}^2 +\g_{14}^2-\g_{24}^2) & 1/2(\g_{13}^2 +\g_{14}^2-\g_{34}^2) & 1/2(\g_{14}^2 +\g_{14}^2-\g_{44}^2)
\end{pmatrix} = 
\begin{pmatrix}
4 & -3/2 & -5/2 \\ \\
-3/2 & 9 & 8\\ \\
-5/2 & 8 & 16
\end{pmatrix}
\] 
\medskip

\noindent One can check that the eigenvalues of $Q$ are approximately 21.7, 3.81, and 3.48.  Then, since $Q$ is positive-definite, $\t$ is a legitimate Euclidean simplex.  

\subsection*{Calculating distances in $(\t, g_\E)$}  Let $p = (\frac{1}{4},\frac{1}{4},\frac{1}{4},\frac{1}{4})$ and $q=(\frac{1}{3},\frac{1}{3},\frac{1}{3},0)$. We wish to calculate $d_\E(p,q)$.  The squared Euclidean distance between them, by Equation \ref{euclid-dist}, is
\[d_\E^2(p,q) = [p-q]^TQ_E[p-q] = 
\begin{pmatrix}
\frac{1}{12} & \frac{1}{12} & \frac{-1}{4}
\end{pmatrix}\begin{pmatrix}
4 & -3/2 & -5/2 \\
-3/2 & 9 & 8\\
-5/2 & 8 & 16
\end{pmatrix}\begin{pmatrix}
\frac{1}{12}\\\frac{1}{12}\\\frac{-1}{4}
\end{pmatrix} = \frac{121}{144},\] so $d_E(p,q)=11/12$.

\subsection*{Projecting $v_1$ onto $\t_1$ in $(\t, g_\E)$}  Now, we find the projection of $v_1$ onto $\t_1$. Firstly, we have
\[
Q_{\t_1} = 
\begin{pmatrix}
1/2(\g_{23}^2 + \g_{23}^2-\g_{33}^2) & 1/2(\g_{23}^2 + \g_{24}^2-\g_{34}^2)\\
1/2(\g_{23}^2 + \g_{24}^2-\g_{34}^2) & 1/2(\g_{24}^2 + \g_{24}^2-\g_{44}^2)
\end{pmatrix} = \begin{pmatrix}
16 & 16 \\
16 & 25
\end{pmatrix}
\] So $|Q_{\t_1}|=144$. We now compute the minors of $Q$:

\[Q_{11}=
\begin{vmatrix}
9 & 8  \\
8 & 16
\end{vmatrix}=80
\qquad Q_{12}=
\begin{vmatrix}
-3/2 & 8  \\
-5/2 & 16
\end{vmatrix}=-4
\qquad Q_{13}=
\begin{vmatrix}
-3/2 & 9  \\
-5/2 & 8
\end{vmatrix}=21/2
\]\[ Q_{22}=
\begin{vmatrix}
4 & -5/2  \\
-5/2 & 16
\end{vmatrix}=231/4
\qquad Q_{23}=
\begin{vmatrix}
4 & -3/2  \\
-5/2 & 8
\end{vmatrix}=113/4
\qquad Q_{33}=
\begin{vmatrix}
4 & -3/2  \\
-3/2 & 9
\end{vmatrix}=135/4
\]
Thus, we can then find each $\alpha_i$, noting that $Q_{ij} = Q_{ji}$ due to the symmetry of $Q$:
\[
\alpha_2 = \frac{Q_{11} - Q_{12} + Q_{13}}{|Q_{\t_1}|} = \frac{80-(-4) + 21/2}{144} \approx 0.65625
\]\[
\alpha_3 = \frac{-Q_{21} + Q_{22} - Q_{23}}{|Q_{\t_1}|} = \frac{-(-4) + 231/4 - 113/4}{144} \approx 0.23264
\]\[
\alpha_4 = \frac{Q_{31}-Q_{32}+Q_{33}}{|Q_{\t_1}|} = \frac{21/2-113/4+135/4}{144} \approx 0.11111
\] Let us quickly remark that the subscripts of the $\a_i's$ are off by 1 from Theorem \ref{thm:Euclid Projection} since we are projecting $v_1$ as opposed to $v_4$ as in the Theorem.  

Note that $\alpha_2+\alpha_3+\alpha_4=1$. So the orthogonal projection of $v_1$ onto $\t_1$ has barycentric coordinates $(0,0.6525, 0.23264, 0.11111)$. It is worth noting that each coordinate is positive, and so the projection $p$ lies inside the triangle $\t_1$.  Our formula gives an easy way to check if a vertex projects inside or outside of the opposite face.  

\subsection*{Calculating the ``height" of $(\t, g_\E)$}
Consider the altitude of $v_1$ over $p$:
\[
d_\E(v_1,p) = \sqrt{p^TQp} \approx 1.4136.
\]
We will use this result to compare to the analogous result in the hyperbolic example.

\subsection{Calculations in $(\t,g_\H)$}

\subsection*{Verifying that $(\t, g_\H)$ is a legitimate hyperbolic simplex}  Via Theorem \ref{theorem : hyperbolic-realizability}, we need to calculate $Q_{\Sigma}$ using equation \eqref{eqn:Qij}.  We have that $Q_\Sigma$ =

\[
\begin{pmatrix}
-1 & -\cosh(\g_{12}) & -\cosh(\g_{13}) & -\cosh(\g_{14})\\
\\
-\cosh(\g_{12}) & -1 & -\cosh(\g_{23}) & -\cosh(\g_{24})\\
\\
-\cosh(\g_{13}) & -\cosh(\g_{23}) & -1 & -\cosh(\g_{34})\\
\\
-\cosh(\g_{14}) & -\cosh(\g_{24}) & -\cosh(\g_{34}) & -1
\end{pmatrix} = \begin{pmatrix}
-1 & -\cosh(2) & -\cosh(3) & -\cosh(4)\\
\\
-\cosh(2) & -1 & -\cosh(4) & -\cosh(5)\\
\\
-\cosh(3) & -\cosh(4) & -1 & -\cosh(3)\\
\\
-\cosh(4) & -\cosh(5) & -\cosh(3) & -1
\end{pmatrix}
\] The eigenvalues of $Q_\Sigma$ are approximately -90.1, 79.2, 5.5, and 1.4.  Since $Q_\Sigma$ has signature $(3,1)$, by Theorem 1 we know that $\t$ is a legitimate hyperbolic simplex.

\subsection*{Calculating distances in $(\t, g_\H)$}  We now wish to calculate $d_\H(p,q)$, where $p = (\frac{1}{4},\frac{1}{4},\frac{1}{4},\frac{1}{4})$ and $q=(\frac{1}{3},\frac{1}{3},\frac{1}{3},0)$. By Theorem \ref{thm:hyp dist}, we have
\[d_\H(p,q) = \arccosh\left(\frac{-\la p,q \ra}{\sqrt{\la p,p \ra \cdot \la q,q \ra}}\right),\]
with:
\[
\la p,q \ra = p^t(Q_\Sigma) q \approx -16.40517\;, \quad \la p,p \ra = p^t(Q_\Sigma) p \approx -19.34049, \text{  and}\quad \la q,q \ra = q^t(Q_\Sigma) q \approx -9.47513.
\]
Thus, we have that \[d_\H(p,q) = \arccosh \left(\frac{16.40516}{\sqrt{(-19.34049)(-9.47513)}}\right) = 0.63997.\]

\medskip

Notice here that $d_\H(p,q) < d_\E(p,q)$, as expected. 

\subsection*{Projecting $v_1$ onto $\t_1$ in $(\t, g_\H)$}
We now consider the projection $\tilde{p}$ of $v_1$ onto $\t_1$. First, we find the barycentric coordinates for the corresponding point $p = (0,\a_2,\a_3,\a_4)$ on the convex hull $\s$, before finding the coordinates of $\tilde{p}$. First, the relevant minors of $Q_\Sigma$ are:

\[
\qquad
Q_{12}^\Sigma = \begin{vmatrix}
-\cosh(2) & -\cosh(4) & -\cosh(5)\\
\\
-\cosh(3) & -1 & -\cosh(3)\\
\\
-\cosh(4) & -\cosh(3) & -1
\end{vmatrix}\approx -12350.57
\]

\medskip

\[
Q_{13}^\Sigma = \begin{vmatrix}
-\cosh(2) & -1 & -\cosh(5)\\
\\
-\cosh(3) & -\cosh(4) & -\cosh(3)\\
\\
-\cosh(4) & -\cosh(5) & -1
\end{vmatrix}\approx 2340.72 \quad 
Q_{14}^\Sigma = \begin{vmatrix}
-\cosh(2) & -1 & -\cosh(4)\\
\\
-\cosh(3) & -\cosh(4) & -1\\
\\
-\cosh(4) & -\cosh(5) & -\cosh(3)
\end{vmatrix}\approx -718.81
\]

\bigskip
\noindent From Theorem \ref{theorem:hyperbolic-projection}, we know that
    \begin{equation*}
    \a_2 = \frac{-Q_{12}}{-Q_{12}+Q_{13}-Q_{14}} \qquad \a_3 = \frac{Q_{13}}{-Q_{12}+Q_{13}-Q_{14}} \qquad \a_4 = \frac{-Q_{14}}{-Q_{12}+Q_{13}-Q_{14}}.  
    \end{equation*}
Plugging in the values for the minors and calculating, we get that $p = (0, 0.80146, 0.15190, 0.04665)$. Note that $\a_2+\a_3+\a_4 = 1$, as expected.
Also by Theorem \ref{theorem:hyperbolic-projection}, we have
\[
\tilde{p} = (0,\tilde{\a}_2,\tilde{\a}_3,\tilde{\a}_4);
\quad \text{for } \tilde{\a}_i = \frac{(-1)^{i+1}Q^\Sigma_{1i}}{\sqrt{\sum_{i,j=2}^{n+1}(-1)^{i+j}Q^\Sigma_{1i}Q^\Sigma_{1j}(\cosh(\g_{ij}))}}
\]And thus, since \[\sqrt{\sum_{i,j=2}^{n+1}(-1)^{i+j}Q^\Sigma_{1i}Q^\Sigma_{1j}(\cosh(\g_{ij}))} \approx 55578.499 \text{, we have  } \tilde{p} = \left(0,0.22222,0.04212, 0.01293  \right).\]  Alternatively, one could just calculate $\la p, p \ra$, and then use that $\tilde{p} = \frac{p}{\sqrt{- \la p, p \ra}}$.  

\subsection*{Calculating the ``height" of $(\t, g_\H)$}
Let us find the altitude of $v_1$ over $\tilde{p}$:
\[ d_\E^2(v_1,\tilde{p}) = \begin{pmatrix} -1 & 0.222 & 0.042 & 0.0129  \end{pmatrix}Q_\Sigma \begin{pmatrix} -1 \\ 0.222 \\ 0.042 \\ 0.0129 \end{pmatrix}\approx 1.22644 \]\[
d_\H(v_1,\tilde{p}) = \arccosh\left(\frac{2+1.22644}{2}\right) \approx 1.0575.
\]
Note that $d_\H(v_1,p) < d_\E(v_1,p)$, as we would expect.

\section{Analogous formulas for spherical simplices}\label{section:spherical}

Our model space for $\S^n$ is the unit sphere in $\R^{n+1}$.  Similar to equation \eqref{eqn:hyperbolic distance} we have
    \begin{equation}\label{eqn:spherical distance}
    d_\S(x,y) = \arccos(\la x,y \ra).
    \end{equation}
Also, in the same way as we did in Section \ref{section:realizability}, we can consider the $(n+1)$-simplex $\Sigma$ which is the convex hull of the vertices of a simplex and the origin.  The Gram matrix $Q_\Sigma$ is calculated by the formula
    \begin{equation}\label{eqn:spherical Qij}
    q_{ij} = \la w_i, w_j \ra = \la v_i, v_j \ra = \cos(\g_{ij}).
    \end{equation}

\begin{theorem}[Analogous to Theorem \ref{theorem : hyperbolic-realizability}]\label{theorem : spherical-realizability}
A collection of $n(n+1)/2$ positive real numbers $\{\gamma_{ij}\}_{i,j=1}^{n+1}$ with $\g_{ij} < \pi/2$ for all $i, j$ are the edge lengths of a spherical $n$-simplex $(\t, g_\mathbb{S})$ if and only if the $(n+1) \times (n+1)$ matrix $Q_\Sigma$ defined by equation \eqref{eqn:spherical Qij} is positive-definite.  
\end{theorem}

    \begin{proof}
    This is essentially identical to the proof of Theorem \ref{theorem : hyperbolic-realizability}.

    \end{proof}
    
Distances in spherical simplices are computed in the analogous way as they are in hyperbolic simplices:  one projects the points onto the sphere using the techniques from Section \ref{section:coordinates}, and then calculates the distance using equation \eqref{eqn:spherical distance} and $Q_{\Sigma}$.  

For orthogonal projection in spherical simplices, you just project within the convex hull of the points and then project that point onto the sphere.  More precisely, let $(\t, g_\S)$ be a spherical simplex.  Linearly isometrically embed $\t$ into $\S^n$ in some way, and identify the vertices $v_i$ with their image in $\S^n$.  Let $\s$ be the $n$-simplex formed by the convex hull of $(v_1, \dots, v_{n+1})$.  Then, to calculate $\text{proj}_{\t_1}(v_1)$, you calculate $\text{proj}_{\s_1}(v_1)$ using Theorem \ref{thm:Euclid Projection} and then project this point onto the sphere from the origin.  
This process works for spherical simplices but not for hyperbolic simplices because the quadratic form restricted to the hyperplane containing $\s$ is always positive-definite for spherical simplices.  

Finally, the distance and projection formulas in this paper can be extended to simplices with constant curvature $\kappa$ by adjusting equations \eqref{eqn:hyperbolic distance}, \eqref{eqn:Qij}, \eqref{eqn:spherical distance}, and \eqref{eqn:spherical Qij} accordingly.

\bibliographystyle{amsplain}

\end{document}